\documentclass{article}

\usepackage{amsfonts}
\usepackage{amsmath}
\usepackage{amsthm}
\usepackage[all]{xy}
\usepackage{mathrsfs}
\usepackage{verbatim}
\usepackage{graphicx}
\usepackage{color}
\usepackage[usenames,dvipsnames]{xcolor}

\theoremstyle{plain}
\newtheorem{theorem}{Theorem}[section]
\newtheorem{lemma}[theorem]{Lemma}

\theoremstyle{definition}

\newtheorem{example}[theorem]{Example}
\newtheorem{remark}[theorem]{Remark}

\numberwithin{equation}{section}

\newcommand{\NN}{{\mathbb N}}
\newcommand{\ZZ}{{\mathbb Z}}
\newcommand{\QQ}{{\mathbb Q}}
\newcommand{\aaa}{{\mathbb A}}

\newcommand{\BB}{\mathcal{B}}
\newcommand{\RR}{{\mathbb R}}
\newcommand{\oo}{\mathbb{O}}

\newcommand{\SC}{\mathcal{C}}

\title{On De Graaf spaces of pseudoquotients}
\author{Anya Katsevich and Piotr Mikusi\'nski\footnote{Corresponding author. Email: piotr.mikusinski@ucf.edu}\\
Department of Mathematics\\
University of Central Florida\\
 Orlando, Florida, USA}

\begin{document}
\maketitle

\begin{abstract}  A space of pseudoquotients $\mathcal{B}(X,S)$ is defined as equivalence classes of pairs $(x,f)$, where $x$ is an element of a non-empty set $X$, $f$ is an element of $S$, a commutative semigroup of injective maps from $X$ to $X$, and $(x,f) \sim (y,g)$ if $gx=fy$. In this note we consider a generalization of this construction where the assumption of commutativity of $S$ by Ore type conditions. As in the commutative case, $X$ can be identified with a subset of $\mathcal{B}(X,S)$ and $S$ can be extended to a group $G$ of bijections on $\mathcal{B}(X,S)$. We introduce a natural topology on $\mathcal{B}(X,S)$ and show that all elements of $G$ are homeomorphisms on $\mathcal{B}(X,S)$. 
\end{abstract}

\bigskip



\section{Introduction}

We are interested in algebraic constructions that can be used in the theory of generalized functions. Jan Mikusi\'nski used the fact that every integral domain can be extended to a field of quotients to construct a space of generalized functions on $[0,\infty)$ (see \cite{JMPol}, \cite{JMEng}, as well as \cite{JGM}). In this case, the integral domain is the space $\SC([0,\infty ))$ of continuous functions on $[0,\infty )$ with the operations of addition and convolution defined as
$$ 
(f*g)(t) = \int_0^t f(s)g(t-s) \,ds.
$$
K\^osaku Yosida modified the construction by considering quotients of the form $\frac{f}{h^n}$, where $f\in \SC([0,\infty ))$ and $h$ is the constant function $1$ on $[0,\infty )$, \cite{Yosida}. Since $h$ is an element of $\SC([0,\infty ))$, the obtained ring of quotients is subset of the field of quotients obtained from $\SC([0,\infty ))$. However, one could take a somewhat different point of view.  We think of $h$ as a map $h:\SC([0,\infty )) \to \SC([0,\infty ))$ defined by
$$
hf(t)= \int_0^t f(s)\,ds
$$
and define a semigroup $S=\{h^n : n=0,1,2, \dots \}$ acting on $\SC([0,\infty ))$.  It turns out that this approach creates new possibilities. In 1987 Jan De Graaf and Tom Ter Elst proposed the following general framework for constructing spaces of generalized functions: a (possibly non-commutative) ring $R$ acting on a vector space $V$, see \cite{dGtE}. There are two essential ideas here: $R$ need not be a subset of $V$ and $R$ need not be commutative.

A similar construction in the commutative case was used in \cite{Quo}. If $X$ is a non-empty set and $S$ is a commutative semigroup of injective maps acting on $X$, then the space of pseudoquotients $\BB(X,S)$ is defined as the space of equivalence classes of pairs $(x,f)\in X \times S$ with respect to the equivalence relation defined as follows: $(x,f) \sim (y,g)$ if $gx=fy$. A construction in the case when elements of $S$ are not injective is also considered.

Pseudoquotients have desirable properties. The set $X$ can be identified with a subset of $\BB(X,S)$ and the semigroup $S$ can be extended to a commutative group of bijections acting on $\BB(X,S)$ (see, for example, \cite{KimPM}).  Under natural conditions on $S$, the algebraic structure of $X$ extends to $\BB(X,S)$. Examples of applications of pseudoquotients in generalized functions and  abstract harmonic analysis can be found in \cite{Levy}, \cite{Radon}, \cite{AMN}, and \cite{AMS}. 

There is interest in applying algebraic methods in noncommutative cases.  In \cite{R-K} we find an interesting application of the standard noncommutative localization in the theory of integro-differential algebras and operators. Noncommutative methods in generalized functions are also considered in \cite{HH}. 

The paper by De Graaf and Ter Elst was published in not easily available conference proceedings and has not received the recognition it deserves. None of the papers mentioned above reference the paper. The authors of this note were not familiar with the results presented in \cite{dGtE} until recently.

In this note we present a slightly more general version of the construction of De Graaf and Ter Elst. We construct pseudoquotients in the case where commutativity of $S$ is replaced by the left Ore condition. We obtain a space $\BB(X,S)$ that contains a copy of $X$. If, in addition to the left Ore condition, we assume right cancellation in $S$, then $S$ can be extended to a group $G$ of bijections acting on $\BB(X,S)$. We introduce a natural topology on $\BB(X,S)$ and show that all elements of $G$ are homeomorphisms on $\BB(X,S)$.  We also give some simple examples. 

\section{Pseudoquotients}

Let $X$ be a nonempty set and let $S$ be a semigroup acting on $X$ injectively. We will assume that $S$ satisfies the following Ore condition:
\begin{itemize}
\item[$\oo$] For any $f,g\in S$ there exist $f',g'\in S$ such that $f'g=g'f$.
\end{itemize} 
In $X\times S$ we introduce an equivalence relation:
\begin{equation*}\label{equiv}
(x,f) \sim (y,g) \text{ if there exist } f',g' \in S \text{ such that } f'g=g'f \text{ and } f'y=g'x.
\end{equation*}

It is easy to verify that $\sim$ is reflexive and symmetric (even without $\oo$). To show that it is also transitive we assume $(x,f) \sim (y,g)$ and $(y,g) \sim (z,h)$.  Let $f', g', g'', h'' \in S$ be such that
$$
f'g=g'f,\; f'y=g'x,\; g''h=h''g, \text{ and } g''z=h''y. 
$$
By $\oo$, there exist $j,k\in S$ such that $jf'=kh''$. Then
$$
jg'f=jf'g=kh''g=kg''h
$$
and
$$
jg'x=jf'y=kh''y=kg''z.
$$
Consequently, $(x,f) \sim (z,h)$. 

We define $\BB = \BB (X,S)= (X\times S)/\sim$.  Elements of $\BB (X,S)$ will be called pseudoquotients.  We will use the standard notation to denote elements of $\BB$: $[(x,f)]=\frac{x}{f}$.

\begin{lemma}\label{lem1} For all $x \in X$ and $f,g\in S$ we have
\begin{itemize}
\item[\rm{(a)}] $(fx,f) \sim (gx,g)$,
\item[\rm{(b)}] If $(fx,f) \sim (y,g)$, then $y=gx$,
\item[\rm{(c)}] If $(fx,f) \sim (fy,f)$, then $x=y$.
\end{itemize}
\end{lemma}

\begin{proof}
 (a) By $\oo$, there are $f',g'\in S$ such that $f'g=g'f$.  Then also $f'gx=g'fx$.
 
 (b) If $(fx,f) \sim (y,g)$, then there are $f',g'\in S$ such that $g'fx=f'y$ and $g'f=f'g$.  Hence
 $$
 f'y=g'fx=f'gx.
 $$
 Since $f'$ is injective, we obtain $y=gx$.
 
 (c) follows from (b), if we take $fy$ in place of $y$ and $f$ in place of $g$ and then use injectivity if $f$.
\end{proof}

Part (a) of Lemma \ref{lem1} implies that the map $\iota : X \to \BB$ defined by
$$
\iota (x)=\frac{fx}{f}
$$ 
is well-defined. From part (b) we have $[(fx,f)]=\{ (gx,g) : g\in S \}$.  This, combined with (c), implies that $\iota$ is injective.

Pseudoquotients have many properties of standard quotients.  The following simple lemma gives an example of such a property.

\begin{lemma} For all $x,y \in X$ and $f,g,h\in S$ we have
\begin{itemize}
\item[\rm{(a)}] If $(x,f) \sim (y,g)$, then $(x,f) \sim (hy,hg)$,
\item[\rm{(b)}] $\frac{x}{f}= \frac{gx}{gf}$.
\end{itemize}
\end{lemma}

\begin{proof}
 (a)  If $(x,f) \sim (y,g)$, then there exist  $f',g' \in S$ such that $g'x=f'y$ and $g'f=f'g$. If $h\in S$, then $\alpha f' = \beta h$ for some $\alpha, \beta \in S$, by $\oo$. Hence 
$$
\alpha g'x =\alpha f'y = \beta hy
$$
and
$$
\alpha g'f= \alpha f' g = \beta hg.
$$

(b) is a direct consequence of (a).
\end{proof}

\section{Extendability of maps}

One of the fundamental properties of pseudoquotients with commutative $S$ is that $S$ can be extended to a group of bijections acting on $\BB$. To obtain such a result without commutativity, in addition to $\oo$, we will assume right cancellation for $S$, that is:

\begin{itemize}
\item[$\aaa$] If $f_1g = f_2g$, where $f_1, f_2, g \in S$, then $f_1 = f_2$.  
\end{itemize}

\begin{lemma}\label{oldA}
 If $f'g=g'f$ and $f''g=g''f$, then $hf''=kf'$ and $hg'' = kg'$ for some $h,k\in S$.
\end{lemma}

\begin{proof}
 Let $f'g=g'f$ and $f''g=g''f$.  Then there are $h,k\in S$ such that $hf''=kf'$, by $\oo$. Hence, $hf''g=kf'g$ and $hg''f=kg'f$, which implies $hg''=kg'$ by $\aaa$.
\end{proof}

\begin{lemma} \label{lemma}
Let $f'g=g'f$ and $f''g=g''f$.  If $f'y = g'x$ for some $x, y \in S$, then $f''y = g''x$.
\end{lemma} 
\begin{proof}
 Let $h,k\in S$ be as defined in Lemma \ref{oldA}. Then we have
$$ hf''y=kf'y=kg'x=hg''x.$$ 
Since $h$ is injective, we obtain $f''y=g''x$. 
\end{proof}

Note that from the above lemma it follows that, if $\frac{x}{f} \sim \frac{y}{g}$, then we have $g'x = f'y$ for any $g', f'\in S$ such that $f'g=g'f$.

\begin{lemma}\label{well-def1}
If $f'g=g'f$ and $f''g=g''f$, then $\frac{g'x}{f'} = \frac{g''x}{f''}$ for all $x\in X$.
\end{lemma}

\begin{proof}
Let $h, k$ be defined as in Lemma \ref{oldA}. Then $\frac{g'x}{f'} = \frac{g''x}{f''}$ because $kf' = hf''$ and $kg'x = hg''x$.
\end{proof}

\begin{theorem}\label{goodfxn} A function $g \in S$ can be extended to a function $\tilde{g}:\BB\to\BB$ by 
$$
\tilde{g}\frac{x}{f} = \frac{g'x}{f'},
$$
where $f',g'\in S$ are any functions such that $f'g=g'f$. 
\end{theorem}

\begin{proof}
By Lemma \ref{well-def1}, $\frac{g'x}{f'}$ is independent of choice of $g', f'$, as long as $f'g=g'f$. 

Now we show that, if $\frac{x_1}{f_1} = \frac{x_2}{f_2}$, then $\tilde{g}\frac{x_1}{f_1} = \tilde{g}\frac{x_2}{f_2}$. Let $f_1',f_2',g',g'' \in S$ be such that
$$
g'f_1=f_1'g \; \text{ and } \; g''f_2=f_2'g.
$$
We need to show that $\frac{g'x_1}{f_1'} \sim \frac{g''x_2}{f_2'}$. Let  $f_1'',f_2''\in S$ be such that $f_1''f_2'=f_2''f_1'$. Then 
$$
f_1''g''f_2 = f_1''f_2'g = f_2''f_1'g = f_2''g'f_1.
$$
Since $\frac{x_1}{f_1} = \frac{x_2}{f_2}$, we obtain $f_1''g''x_2 = f_2''g'x_1$, by Lemma \ref{lemma}.  

We have shown that $\tilde{g}$ is well defined. Finally we show that $\tilde{g}$ is an extension of $g$. If $f'g=g'f$, then $f'gx=g'fx$ and hence
$$
\tilde{g} \iota(x) =  \tilde{g}\frac{fx}{f} = \frac{g'fx}{f'} = \frac{f'gx}{f'}=\iota(gx).
$$
\end{proof}

\begin{lemma} 
$\tilde{g} : \BB \to \BB$ is bijective for every $g\in S$. 
\end{lemma}

\begin{proof} Injectivity: Suppose $\tilde{g}(\frac{x_1}{f_1}) = \tilde{g}(\frac{x_2}{f_2})$ and let $f_1', g', f_2', g'' \in S$ be such that 
$$f_1'g = g'f_1 \; \text{ and } \; f_2'g = g''f_2.$$
 Then $\tilde{g}(\frac{x_1}{f_1}) = \frac{g'x_1}{f_1'}$ and $\tilde{g}(\frac{x_2}{f_2}) = \frac{g''x_2}{f_2'}$. Now let $h_1, h_2 \in S$ be such that $h_2f_1' = h_1f_2'$. Then $h_2g'x_1 = h_1g''x_2$, by Lemma \ref{lemma}. We have
$$\frac{x_1}{f_1} = \frac{g'x_1}{g'f_1} = \frac{g'x_1}{f_1'g} = \frac{h_2g'x_1}{h_2f_1'g}$$
and
$$\frac{x_2}{f_2} = \frac{g''x_2}{g''f_2} = \frac{g''x_2}{f_2'g} = \frac{h_1g''x_2}{h_1f_2'g}.$$ 
Since $\frac{h_1g''x_2}{h_1f_2'g} = \frac{h_2g'x_1}{h_2f_1'g}$, we obtain $\frac{x_1}{f_1} = \frac{x_2}{f_2}$ by transitivity.

Surjectivity: If $\frac{x}{f} \in B$, then $\frac{x}{f}=\frac{fx}{f^2} = \tilde{g}(\frac{x}{fg})$ since $f(fg) = f^2(g)$. 
\end{proof}

Note that $\tilde{g}^{-1}(\frac{x}{f}) = \frac{x}{fg}$.

{\begin{theorem}
 $S$ can be extended to a group $G$ of bijections acting on $\BB$.  Moreover, $G$ satisfies conditions $\oo$ and $\aaa$.
\end{theorem}

\begin{proof}
It suffices to show that $G$ satisfies conditions $\oo$ and $\aaa$, but this follows immediately from the fact that everything in $G$ is invertible.
\end{proof}

\begin{remark} 
We can think of $\BB$ as the set of all solutions $\xi$ of the equations $f(\xi) = x$, for any $f \in S$ and $x \in X$. These solutions are unique, for if $\tilde{f}(\frac{x_1}{f_1}) =x$ and  $\tilde{f}(\frac{x_2}{f_2}) =x$, then $\frac{x_1}{f_1} = \frac{x_2}{f_2}$, since $\tilde f$ acts injectively on $\BB$. Also any element of $\BB$ is a solution to some equation, since $\tilde{f}(\frac{x}{f}) = \frac{fx}{f} = \iota (x)$. \end{remark}

\section{Topology of pseudoquotients}

Now we assume that $X$ is a topological space and $S$ is a semigroup of continuous injections acting on $X$.  To define a topology in $\BB$ we equip $S$ with the discrete topology, then $X \times S$ with the product topology, and finally $\BB = X\times S / \sim$ with the quotient topology. This topology has desirable properties.  In particular, the embedding $\iota : X \to \BB$ is continuous and $G$ is a group of homeomorphisms acting on $\BB $. 

\begin{theorem}
 $\iota$ is continuous.
\end{theorem}

\begin{proof}
Let $U$ be open in $\BB$, and let $x\in \iota^{-1}(U)$. Then $ \frac{fx}{f} \in U$. Since $U$ is open in $\BB$, we have that $fx \times f \in V \times f \subset p^{-1}(U)$, where $V$ is open in $X$ and $p:X \times S \to \BB$ is the quotient map defined by $p(x \times f) = \frac{x}{f}$. Then $x \in f^{-1}(V) \subset \iota^{-1}(U)$, which shows that $\iota^{-1}(U)$ is open.
\end{proof}
\begin{theorem}
Elements of the group $G$ are homeomorphism.
\end{theorem}

\begin{proof}
Consider the map $\psi_g : X \times S \to \BB$ defined by $\psi_g(x \times f) = \frac{g'x}{f'}$, where $f'g = g'f$. Since the function $\tilde g$ is well-defined, so is $\psi_g$. Further, it is clear that $\psi_g(p^{-1}(\{\frac{x}{f}\})) = \frac{g'x}{f'}$. It follows that, if $\psi_g$ is continuous, then $\tilde g$ is also continuous. Let $U$ be open in $\BB$, and let $x \times f \in \psi_g^{-1}(U)$. Then $ \frac{g'x}{f'} \in U$. Since $U$ is open in $\BB$, we have that $g'x \times f' \in V \times f' \subset p^{-1}(U)$, where $V$ is open in $X$. Then it is easy to show that $x \times f \in g'^{-1}(V) \times f \subset \psi_g^{-1}(U)$, from which it follows that $\psi_g^{-1}(U)$ is open in $X\times S$ and that $\psi_g$ is continuous. 

To show that $\tilde g^{-1}$ is continuous, we similarly show that $\psi_{g^{-1}} : X \times S \to \BB$ defined by $\psi_g(x \times f) = \frac{x}{fg}$ is continous. 
\end{proof}

\section{Simple examples}

We end this paper with some examples of the described construction.

\begin{example}
Let $X = \NN$ and let $S$ be the semigroup of all functions on $\NN$ of the form $f(x)=mx^n$ where  $m, n \in \NN$. It is easy to see that $\oo$ and $\aaa$ are satisfied. In this case $\BB$ can be identified with the set of all real numbers of the form $\sqrt[n]{\frac{k}{m}}$, where $k,m,n\in \NN$, and the extended group $G$ can be described as all functions of the form $f(x)=px^q$, where $p$ and $q$ are positive rational numbers.
\end{example}

\begin{example} Let $X = \ZZ^n$ and let $S$ be the semigroup of all functions $f : \ZZ^n \to \ZZ^n$ of the form $f(x) = Mx+b$, where $M$ is an $n \times n$ matrix of rank $n$ with integer entries and $b\in\ZZ^n$.   To check $\oo$, let $f(x) = M_1x + b_1$ and $g(x) = M_2x + b_2$. If $m_1 = \det M_1$ and $m_2 = \det M_2$, then the functions 
$$f'(x) = m_1m_2M_2^{-1}x + m_1m_2M_1^{-1}b_1$$
and
$$g'(x) = m_1m_2M_1^{-1}x + m_1m_2M_2^{-1}b_2$$ 
satisfy the equation $f'g = g'f$. It is easy to see that $\aaa$ is also satisfied. 

Here $\BB$ can be identified with $\QQ^n$ and the extended group $G$ can be described as all functions $f : \QQ^n \to \QQ^n$ of the form $f(x) = Mx+b$, where $M$ is an $n \times n$ invertible matrix with rational entries and $b\in\QQ^n$.
\end{example}}

\begin{example}
 Let
$$
X=\left\{\sum_{k=0}^n \lambda_k \chi_{[k, k+1)}: \lambda_k\in \RR \text{ and }n\in \NN  \right\}
$$
and let $\delta , \tau : X\to X$ be defined as follows:
$$
\delta \left( \sum_{k=0}^n\lambda_k \chi_{[k, k+1)} \right)= \sum_{k=0}^n \frac{\lambda_k}2 \chi_{[2k, 2k+2)}
\quad \text{and} \quad 
\tau \left( \sum_{k=0}^n\lambda_k \chi_{[k, k+1)} \right)= \sum_{k=0}^n\lambda_k \chi_{[k+1, k+2)}.
$$
If we define $I=\chi_{[0,1)}$, then we can write
$$
\sum_{k=0}^n\lambda_k \chi_{[k, k+1)} = \sum_{k=0}^n\lambda_k \tau^k I.
$$ 
Let $S$ be the semigroup generated by $\delta$ and $\tau$. Since $\delta \tau = \tau^2 \delta$, we have $S=\left\{ \tau^m \delta^n : m,n \geq 0 \right\}$. Since $\delta$ and $\tau$ are injective, $S$ acts on $X$ injectively.  We will show that $S$ satisfies $\oo$ and $\aaa$. To verify $\oo$, assume $\tau^{m_1} \delta^{n_1},\tau^{m_2} \delta^{n_2} \in S$. If $n_1=n_2=n$, then 
$$
\tau^{m_2}\tau^{m_1} \delta^n=\tau^{m_1}\tau^{m_2} \delta^n.
$$
If $n_1<n_2$, then 
$$
\tau^{m_2}\delta^{n_2-n_1}\tau^{m_1} \delta^{n_1}=\tau^{2m_1}\tau^{m_2} \delta^{n_2}.
$$
To verify $\aaa$, first observe that for every $\phi \in S$ the representation $\phi =\tau^m \delta^n$ is unique.  Now, if
$$
\tau^{m_1} \delta^{n_1}\tau^k \delta^l=\tau^{m_2} \delta^{n_2}\tau^k \delta^l,
$$
we have
$$
\tau^{m_1+2k} \delta^{n_1+l}=\tau^{m_2+2k} \delta^{n_2+l}.
$$
Thus $m_1=m_2$ and $n_1=n_2$.

Note that, since $\delta$ and $\tau$ preserve the integral and the $L^1$-norm, we can define
$$
\int \frac{f}{\tau^m \delta^n}= \int f \quad \text{and} \quad \left\| \frac{f}{\tau^m \delta^n}\right\|_1 = \| f \|_1,
$$
where $f\in X$. The obtained space of pseudoquotients $\BB$ can be identified with a dense subspace of $L^1(\RR)$.
\end{example}

\begin{example}
Let $X_1,X_2,\dots$ be a sequence of nonempty sets of the same cardinality. For every $n\in\NN$, let $\varphi_n : X_n \to X_{n+1}$ be a bijection and let $\psi_n :X_n\to X_n$ be an injection such that the following diagram commutes:
\begin{displaymath}
\xymatrix{X_n \ar[r]^{\psi_n} \ar[d]_{\varphi_n} & X_n \ar[d]^{\varphi_n} \\X_{n+1} \ar[r]_{\psi_{n+1}} & X_{n+1}}
\end{displaymath}
Let $X=\cup_{n=1}^\infty X_n\times \{n\}$ and let $\Phi : X \to X$ be the function defined by 
$$
\Phi ((x,n)) = (\varphi_n(x),n+1).
$$
Let $\Psi_n : X \to X$ be defined by 
$$
\Psi_n((x,m)) = \left\{ \begin{array}{ll}
(\psi_n(x),n) &\mbox{ if $m=n$,} \\
(x,m) &\mbox{ otherwise.}
\end{array} \right.
$$
Since $\Phi \Psi_n = \Psi_{n+1} \Phi$ for every $n\in\NN$ and the maps $\Psi_n$ commute, the semigroup generated by $\Phi, \Psi_1, \Psi_2, \dots$ is
$$
S=\left\{\Psi_1^{k_1} \dots \Psi_m^{k_m} \Phi^n : k_1, \dots , k_m, m, n \in \NN \cup \{0\}   \right\}.
$$ 
It is easy to check that $S$ satisfies $\oo$ and $\aaa$.

\end{example}

{\bf Acknowledgements.} The authors would like to thank Joseph Brennan and Jan De Graaf for helpful comments.

\end{document}